\documentclass{amsart}
\usepackage[OT2,T1]{fontenc}
\usepackage{amssymb}
\usepackage{graphicx}
\usepackage{txfonts}
\usepackage{calrsfs}
\usepackage{hyperref}
\newtheorem{theorem}{Theorem}
\newtheorem{corollary}{Corollary}

\newtheorem{lemma}{Lemma}
\newtheorem{proposition}{Proposition}
\newtheorem{remark}{Remark}
\newcommand{\eps}{\varepsilon}
\DeclareMathAlphabet{\mathpzc}{OT1}{pzc}{m}{it}

\DeclareMathOperator{\E}{E}

\DeclareMathOperator{\essinf}{ess\;inf}

\DeclareMathOperator{\re}{Re}

\DeclareMathOperator{\fix}{Fix}
\DeclareMathOperator{\F}{F}
\DeclareMathOperator{\h}{H}
\DeclareMathOperator{\Ka}{k}
\DeclareMathOperator{\co}{co}

\DeclareMathOperator{\im}{Im}

\newcommand{\w}{\tilde}
\newcommand{\map}{\multimap}
\newcommand{\<}{\leqslant}

\newcommand{\n}{{n\geqslant 1}}

\newcommand{\K}{{k\geqslant 1}}
\newcommand{\R}[1]{\varmathbb{R}^{#1}}

\begin{document}

\title[On the properties of the solution set map to Volterra integral inclusion]{On the properties of the solution set map to Volterra integral inclusion}
\author{Rados\l aw Pietkun}
\address{Toru\'n, Poland}
\email{rpietkun@pm.me}

\begin{abstract}
For the multivalued Volterra integral equation defined in a Banach space, the set of solutions is proved to be $R_\delta$, without auxiliary conditions imposed in \cite[Theorem 6.]{pietkun}. It is shown that the solution set map, corresponding to this Volterra integral equation, possesses a continuous singlevalued selection. The image of a convex set under solution set map is acyclic. The solution set to Volterra integral inclusion in a separable Banach space and the preimage of this set through the Volterra integral operator are shown to be absolute retracts.  
\end{abstract}
\maketitle
\section{Introduction}
In \cite{pietkun}, the author conducted the study of geometric properties of the solution set to the following Volterra integral inclusion
\begin{equation}\label{inclusion}
x(t)\in h(t)+\int_0^tk(t,s)F(s,x(s))\,ds,\;t\in I=[0,T]
\end{equation}
in a Banach space $E$. In inclusion \eqref{inclusion}, $h\in C(I,E)$, $k(t,s)\in{\mathcal L}(E)$ and $F\colon I\times E\map E$ was a convex valued perturbation. It was proven that the solution set $S_{\!F}^p(h)$ of integral inclusion \eqref{inclusion} is acyclic in the space $C(I,E)$ or is even $R_\delta$, provided some additional conditions on the Banach space $E$ or the kernel $k$ and perturbation $F$ are imposed. We will show that these auxiliary assumptions are redundant. Thus, the present work complements \cite{pietkun} by strengthening Theorems 6. and 7. In Section 3, we give also some applications of investigated properties of the solution set to problem \eqref{inclusion}. One of them is the characterization of the solution set of an evolution inclusion in the so-called parabolic case. The second is the result on the existence of periodic trajectories of the integral inclusion under consideration.\par Introducing the so-called {\em solution set map} $S_{\!F}^p\colon C(I,E)\map C(I,E)$, which associates with each inhomogeneity $h\in C(I,E)$ the set of all solutions to \eqref{inclusion}, we prove that under generic assumptions on $F$ this multimap possesses a continuous singlevalued selection. In this aim we adapt a well-known construction taken from \cite{colo}.\par Already singlevalued examples show that no more than connectedness can be expected about the image $f(M)$ of a connected set $M$ under continuous $f$. Also in the case of the solution set map is clear only that the set $\bigcup_{h\in M}S_{\!F}^p(h)$ is connected, if $M\subset C(I,E)$ is connected. We exploit the admissibility of the solution set map and the result of Vietoris to demonstrate that the image of a compact convex $M\subset C(I,E)$ through $S_{\!F}^p$ must be acyclic.\par Since the solutions of inclusion \eqref{inclusion} are understood in the sense of Aumann integral, it is natural to examine the issue of geometric structure of the set of these integrable selections of perturbation $F$, which make up the solution set $S_{\!F}^p(h)$ being mapped by the Volterra integral operator. It has been shown that these selections form a retract of the space $L^p(I,E)$. In the context of stronger assumptions about Volterra integral operator kernel, the solution set $S_{\!F}^p(h)$ turns out to be also an absolute retract.
\section{Preliminaries}
Denote by $I$ the interval $[0,T]$ and by $\Sigma$ the $\sigma$-algebra of Lebesgue measurable subsets of $I$. Let $E$ be a real Banach space with the norm $|\cdot|$ and ${\mathcal B}(E)$ the family of Borel subsets of $E$. The space of bounded linear endomorphisms of $E$ is denoted by ${\mathcal L}(E)$ and $E^*$ stands for the normed dual of $E$. Given $S\in{\mathcal L}(E)$, $||S||_{\mathcal L}$ is the norm of $S$. The closure and the closed convex envelope of $A\subset E$ will be denoted by $\overline{A}$ and $\overline{\co} A$ and if $x\in E$ we set $d(x,A)=\inf\{|x-y|\colon y\in A\}$. Besides, for two nonempty closed bounded subsets $A, B$ of $E$ we denote by $d_H(A,B)$ the Hausdorff distance from $A$ to $B$, i.e. $d_H(A,B)=\max\{\sup\{d(x,B)\colon x\in A\},\sup\{d(y,A)\colon y\in B\}\}$.\par By $(C(I,E),||\cdot||)$ we mean the Banach space of continuous maps $I\to E$ equipped with the maximum norm. Let $1\<p<\infty$. Then $(L^p(I,E),||\cdot||_p)$ is the Banach space of all (Bochner) $p$-integrable maps $w\colon I\to E$, i.e. $w\in L^p(I,E)$ iff f is strongly measurable and \[||w||_p=\left(\int_0^T|w(t)|^p\,dt\right)^{\frac{1}{p}}<\infty.\] Notice that strong measurability is equivalent to the usual measurability in case $E$ is separable.\par Recall that a subset $K$ of $L^1(I,E)$ is called decomposable if for every $u, v\in K$ and every $A\in\Sigma$, we have $u{\bf 1}_A+v{\bf 1}_{I\setminus A}\in K$, where ${\bf 1}_A$ stands for the characteristic function of $A$.
\par A set-valued map $F\colon E\map E$ assigns to any $x \in E$ a nonempty subset $F(x)\subset E$. $F$ is upper (lower) semicontinuous, if the small inverse image $F^{-1}(A)=\{x\in E\colon F(x)\subset A\}$ is open (closed) in $E$ whenever $A$ is open (closed) in $E$. A map $F\colon E\map E$ is upper hemicontinuous if for each $p\in E^*$, the function $\sigma(p,F(\cdot))\colon E\to\R{}\cup\{+\infty\}$ is upper semicontinuous as an extended real function, where $\sigma(p,F(x))=\sup_{y\in F(x)}\langle p,y\rangle$. Let $X$ be a separable metric space and ${\mathcal A}$ be a $\sigma$-algebra of subsets of $X$. The map $F\colon X\map E$ is said to be ${\mathcal A}$-measurable if for every open $C\subset X$, we have $F^{-1}(C)\in{\mathcal A}$. A function $f\colon X\to E$ such that $f(x)\in F(x)$ for every $x\in X$ is called a selection of $F$.\par By $H_*(\cdot)$ we denote the \v Cech homology functor with coefficients in the field of rational numbers $\varmathbb{Q}$ from the category of compact pairs of metric spaces and continuous maps of such pairs to the category of graded vector spaces over $\varmathbb{Q}$ and linear maps of degree zero. The space $X$ having the property
\[H_q(X)=
\begin{cases}
0&\mbox{for }q\geqslant 1,\\
\varmathbb{Q}&\mbox{for }q=0
\end{cases}\]
is called acyclic. In other words its homology are exactly the same as the homology of a one point space. A compact (nonempty) space $X$ is an $R_\delta$-set if there is a decreasing sequence of contractible compacta $(X_n)_\n$ containing $X$ as a closed subspace such that $X=\bigcap_\n X_n$. From the continuity of the homology functor $H_*(\cdot)$ follows that $R_\delta$-sets are acyclic.\par We say that a set-valued map $F\colon E\map M$, where $M$ is a metric space, is a $J$-map , if $F$ is upper semicontinuous and the set $F(x)$ is $R_\delta$. An upper semicontinuous map $F\colon E\map M$ is called acyclic if it has compact acyclic values. Let $U$ be an open subset of $E$. A map $\Phi\colon U\map E$ is called decomposable, if there is a $J$-map $F\colon U\map M$, with $M$ being a metric ANR, and a single-valued continuous $f\colon M\to E$ such that $\Phi=f\circ F$. A set-valued map $F\colon E\map M$ is admissible (compare \cite[Def.40.1]{gorn}) if there is a metric space $X$ and two continuous functions $p\colon X\Rightarrow E$, $q\colon X\to M$ from which $p$ is a Vietoris map such that $F(x)=q(p^{-1}(x))$ for every $x\in E$. Clearly, every $J$-map, acyclic map or decomposable map is admissible.\par Finally, a real function $\chi$ defined on the family of bounded subsets of $E$ is called the Hausdorff measure of noncompactness if \[\chi(\Omega):=\inf\{\eps>0:\Omega\mbox{ admits a finite covering by balls of a radius }\eps\}.\] \par Assume that $p$ is a real number from the interval $[1,\infty)$. Throughout this paper it is assumed that $q\in(1,\infty]$ satisfies $q^{-1}+p^{-1}=1$. We will say that a multimap $F\colon I\times E\map E$ satisfies $({\bf F})$ if the following hypotheses are satisfied:
\begin{itemize}
\item[$(\F_1)$] for every $(t,x)\in I\times E$ the set $F(t,x)$ is nonempty, closed and convex,
\item[$(\F_2)$] the map $F(\cdot,x)$ has a strongly measurable selection for every $x\in E$,
\item[$(\F_3)$] the map $F(t,\cdot)$ is upper hemicontinuous for a.a. $t\in I$,
\item[$(\F_4)$] there is $c\in L^p(I,\R{})$ such that $||F(t,x)||^+=\sup\{|y|\colon y\in F(t,x)\}\<c(t)(1+|x|)$ for a.a. $t\in I$ and for all $x\in E$,
\item[$(\F_5)$] there is a function $\eta\in L^p(I,\R{})$ such that for all bounded subsets $\Omega\subset E$ and for a.a. $t\in I$ the inequality holds \[\chi(F(t,\Omega))\<\eta(t)\chi(\Omega).\]
\end{itemize}
We shall say that $F\colon I\times E\map E$ fulfills $({\bf H})$ if the following assumptions are satisfied:
\begin{itemize}
\item[$(\h_1)$] the set $F(t,x)$ is nonempty closed and bounded for every $(t,x)\in I\times E$, 
\item[$(\h_2)$] the map $F(\cdot,x)$ is $\Sigma$-measurable for every $x\in E$,
\item[$(\h_3)$] there exists $\alpha\in L^p(I,\R{})$ such that $d_H(F(t,x),F(t,y))\<\alpha(t)|x-y|$, for all $x,y\in E$, a.e. in $I$,
\item[$(\h_4)$] there exists $\beta\in L^p(I,\R{})$ such that $d(0,F(t,0))\<\beta(t)$, a.e. in $I$.
\end{itemize}
\par Denote by $\bigtriangleup$ the set $\{(t,s)\in I\times I\colon 0\<s\<t\<T\}$. We shall also assume that the mapping $k\colon\bigtriangleup\to{\mathcal L}(E)$ possesses the following properties:
\begin{itemize}
\item[$(\Ka_1)$] the function $k(\cdot,s)\colon[s,T]\to{\mathcal L}(E)$ is differentiable for every $s\in I$,
\item[$(\Ka_2)$] the function $k(t,\cdot)\colon[0,t]\to{\mathcal L}(E)$ is continuous for all $t\in I$,
\item[$(\Ka_3)$] the function $k(\cdot,\cdot)\colon\{(t,t)\colon t\in I\}\to{\mathcal L}(E)$ is continuous, whereas the operator $k(t,t)$ is invertible for all $t\in I$,
\item[$(\Ka_4)$] there exists $\mu\in L^q(I,\R{})$ such that for every $(t,s)\in\bigtriangleup$ we have $\left\Arrowvert\frac{\partial}{\partial t}k(t,s)\right\Arrowvert_{\mathcal L}\<\mu(s)$,
\item[$(\Ka_5)$] for every $t\in I$, $k(t,\cdot)\in L^q([0,t],{\mathcal L}(E))$,
\item[$(\Ka_6)$] the function $I\ni t\mapsto k(t,\cdot)\in L^q([0,t],{\mathcal L}(E))$ is continuous in the norm $||\cdot||_q$ of the space $L^q(I,{\mathcal L}(E))$.
\end{itemize}
\par  By a solution of the Volterra integral inclusion \eqref{inclusion} we mean a function $x\in C(I,E)$, which satisfies equation
\[x(t)=h(t)+\int_0^tk(t,s)w(s)\,ds,\;t\in I\]
for some $w\in L^p(I,E)$ such that $w(t)\in F(t,x(t))$ for a.a. $t\in I$. Obviously, the set of all solutions to integral inclusion under consideration coincides with the set of fixed points $\fix(h+V\circ N_F^p(\cdot))$ of the multivalued operator $h+V\circ N_F^p\colon C(I,E)\map C(I,E)$, where  $N_F^p\colon C(I,E)\map L^p(I,E)$ is the Nemtyskij operator corresponding to $F$, given by
\[N_F^p(x)=\{w\in L^p(I,E)\colon w(t)\in F(t,x(t))\mbox{ for a.a. }t\in I\}\]
and $V\colon L^p(I,E)\to C(I,E)$ is the Volterra integral operator, defined by
\[V(w)(t)=\int_0^t k(t,s)w(s)ds,\;t\in I.\]
The eponymous {\em solution set map} is the multivalued operator $S^p_{\!F}\colon C(I,E)\map C(I,E)$ given by the formula
\[S^p_{\!F}(h):=\left\{x\in C(I,E)\colon x\in h+V\circ N_F^p(x)\right\}=\fix(h+V\circ N_F^p(\cdot)).\]
\par The core reasoning, which support proofs of the results in \cite{pietkun}, on both the topological properties (nonemptiness, compactness) and geometric (acyclicity) of the solution set $S^p_{\!F}(h)$, constitutes the so-called convergence theorem for convex valued upper hemicontinuous multimaps. In order to improve these results, we will need the following ``multivalued version'' of this theorem.
\begin{theorem}\label{convergence}
Assume that $F\colon E\map E$ is upper hemicontinuous and $G\colon I\map E$ has compact values. If sequences $(G_n\colon I\map E)_\n$ and $(y_n\colon I\to E)_\n$ satisfy the following conditions 
\begin{itemize}
\item[(i)] $d_H(G_n(t),G(t))\underset{n\to\infty}{\longrightarrow}0$ for a.a. $t\in I$,
\item[(ii)] $y_n\rightharpoonup y$ in the space $L^p(I,E)$, where $p\geqslant 1$,
\item[(iii)] $y_n(t)\in\overline{\co}B(F(B(G_n(t),\eps_n)),\eps_n)$ for a.a. $t\in I$, where $\eps_n\to 0^+$ as $n\to\infty$,
\end{itemize}
then $y(t)\in\overline{\co}F(G(t))$ for a.a. $t\in I$.
\end{theorem}
\begin{proof}
There is a subset $I_1$ of full measure in $I$ such that (i) and (iii) hold. Take $t\in I_1$. Let $\eps>0$ and $e^*\in E^*\setminus\{0\}$ be arbitrary. Using the upper hemicontinuity of $F$ we see that for every $x\in G(t)$ there is $\delta_x>0$ such that
\[\sigma(e^*,F(B(x,\delta_x)))<\sigma(e^*,F(x))+\frac{\eps}{2}.\] Since $G(t)$ is compact, we have $x_1,\ldots,x_m\in G(t)$ such that $G(t)\subset\bigcup_{i=1}^m B(x_i,\delta_{x_i}/2)$. Put $\delta:=\min\limits_{1\<i\<m}\delta_{x_i}/2$. If $z\in B(G(t),\delta)$, then there is $i\in\{1,\ldots,m\}$ such that \[\sigma(e^*,F(z))<\sigma(e^*,F(x_i))+\frac{\eps}{2}.\] Thus
\[\sigma(e^*,F(B(G(t),\delta)))<\sigma(e^*,F(G(t)))+\frac{\eps}{2}.\] In view of (i) there is an index $N$ such that 
\[B(F(B(G_n(t),\eps_n)),\eps_n)\subset B\left(F(B(G(t),\delta)),\frac{\eps}{2||e^*||}\right)\] for every $n\geqslant N$. Whence
\begin{align*}
\sigma(e^*,&B(F(B(G_n(t),\eps_n)),\eps_n))\<\sigma\left(e^*,B\left(F(B(G(t),\delta)),\frac{\eps}{2||f^*||}\right)\right)\\&=\sigma\left(e^*,F(B(G(t),\delta))+B\left(0,\frac{\eps}{2||e^*||}\right)\right)\<\sigma(e^*,F(B(G(t),\delta)))+\sigma\left(e^*,\frac{\eps}{2||e^*||}\overline{B}(0,1)\right)\\&\<\sigma(e^*,F(G(t)))+\frac{\eps}{2}+\frac{\eps}{2||e^*||}\sigma(e^*,\overline{B}(0,1))=\sigma(e^*,F(G(t)))+\eps
\end{align*}
for $n\geqslant N$. Since $\eps>0$ was arbitrary we see that
\begin{equation}\label{equ1}
\overline{\lim\limits_{n\to\infty}}\sigma(e^*,B(F(B(G_n(t),\eps_n)),\eps_n))\<\sigma(e^*,F(G(t))).
\end{equation}
Since $y_n\rightharpoonup y$ in $L^p(I,E)$, there is a sequence $(z_n)_\n$ strongly convergent to $y$ such that $z_n\in\co\{y_k\}_{k=n}^\infty$. Further, there is a subset $I_2$ of full measure in $I$ and a subsequence (again denoted by) $(z_n)_\n$ pointwise convergent to $y$ for every $t\in I_2$. Let us take $t\in I_1\cap I_2$. Then 
\[z_n(t)\in\co\{y_k(t)\}_{k=n}^\infty\subset\co\bigcup\limits_{k=n}^\infty\overline{\co}B(F(B(G_k(t),\eps_k)),\eps_k)\subset\overline{\co}\bigcup\limits_{k=n}^\infty B(F(B(G_k(t),\eps_k)),\eps_k)\] and
\[\langle e^*,z_n(t)\rangle\<\sigma(e^*,\overline{\co}\bigcup\limits_{k=n}^\infty B(F(B(G_k(t),\eps_k)),\eps_k))=\sup\limits_{k\geqslant n}\sigma(e^*,B(F(B(G_k(t),\eps_k)),\eps_k)).\] As a result we get
\begin{align*}
\langle e^*,y(t)\rangle&=\lim\limits_{n\to\infty}\langle e^*,z_n(t)\rangle\<\lim\limits_{n\to\infty}\sup\limits_{k\geqslant n}\sigma(e^*,B(F(B(G_k(t),\eps_k)),\eps_k)\\&=\overline{\lim\limits_{n\to\infty}}\sigma(e^*,B(F(G_n(t),\eps_n)),\eps_n)).
\end{align*}
Combining this estimation with \eqref{equ1} we see that 
\[\langle e^*,y(t)\rangle\<\sigma(e^*,F(G(t))).\] Bearing in mind that closed convex sets possess the following description \[\overline{\co}F(G(t))=\{v\in E\colon\langle e^*,v\rangle\<\sigma(e^*,\overline{\co}F(G(t)))\mbox{ for every }e^*\in E^*\},\] we conclude finally that $y(t)\in\overline{\co}F(G(t))$ for almost all $t\in I$.
\end{proof}
Taking into account that the initial set of assumptions $({\bf F})$ ensures that the fixed point set $\fix(h+V\circ N_F^p(\cdot))$ is acyclic, one can reverse the problem by asking the question about the structure of $\fix(N_F^p(h+V(\cdot)))$. Thus, it is quite natural to define the {\em selection set map} ${\mathcal S}^p_{\!F}\colon C(I,E)\map L^p(I,E)$, corresponding to the problem \eqref{inclusion}, by \[{\mathcal S}^p_{\!F}(h):=\bigcup_{x\in S^p_{\!F}(h)}\left\{w\in N_F^p(x)\colon x=h+V(w)\right\}=\fix\left(N_F^p(h+V(\cdot))\right).\]
\begin{remark}
Conditions $(\Ka_1)$-$(\Ka_4)$ imply that $\left\{w\in N_F^p(x)\colon x=h+V(w)\right\}$ is a singleton for every $x\in S^p_{\!F}(h)$ $($compare Lemma 2. in \cite{pietkun}$)$.
\end{remark}
Theorems concerning the geometrical structure of the set ${\mathcal S}^p_{\!F}(h)$ are direct conclusions of the following technical lemma.
\begin{lemma}\label{contraction}
Let $E$ be a separable Banach space. Assume that $F\colon I\times E\map E$ satisfies $({\bf H})$ and $k\colon\bigtriangleup\to{\mathcal L}(E)$ satisfies $(\Ka_5)$-$(\Ka_6)$. There is an equivalent norm on the space $L^p(I,E)$, in which the multivalued map $G_p\colon C(I,E)\times L^p(I,E)\map L^p(I,E)$ given by the formula 
\[G_p(h,u):=\left\{w\in L^p(I,E)\colon w(t)\in F\left(t,h(t)+\int_0^tk(t,s)u(s)\,ds\right)\mbox{ a.e. in }I\right\}=N_F^p(h+V(u))\]
is continuous and contractive with respect to $u$.
\end{lemma}
\begin{proof}
Endow the space $L^p(I,E)$ of Bochner $p$-integrable functions with the following equivalent norm 
\begin{equation}\label{norm}
|||w|||_p:=\left(\int_0^Te^{-2^{2p-1}Mr(t)}|w(t)|^p\,dt\right)^{\frac{1}{p}},
\end{equation}
where $M:=\max\left\{1,\sup_{t\in I}||k(t,\cdot)||_q^p\right\}$ and $r(t):=\int_0^t\alpha(s)^p\,ds$.
\par It follows from $(\h_3)$ that $F(t,\cdot)$ is Hausdorff continuous. By virtue of \cite[Th.3.3]{papa} the map $F$ is $\Sigma\otimes{\mathcal B}(E)$-measurable. In particular $F(\cdot,h(\cdot)+V(u)(\cdot))$ is measurable for $(h,u)\in C(I,E)\times L^p(I,E)$, since $F$ is superpositionally measurable (\cite[Th.1.]{zyg}). Thanks to $(\h_3)$ and $(\h_4)$ we have $\inf\{|z|\colon z\in F(t,h(t)+V(u)(t))\}\in L^p(I,\R{})$, which means that $G_p(h,u)\neq\varnothing$.\par Take $(h_1,u_1), (h_2,u_2)\in C(I,E)\times L^p(I,E)$, $w_1\in G_p(h_1,u_1)$ and $\eps>0$. It is easy to see that $t\mapsto d(w_1(t),F(t,h_2(t)+V(u_2)(t)))$ is measurable. By $(\h_3)$ we have \[d(w_1(t),F(t,h_2(t)+V(u_2)(t)))\<\alpha(t)||h_1-h_2+V(u_1-u_2)||\] a.e. in $I$, and so $d(w_1(\cdot),F(\cdot,h_2(\cdot)+V(u_2)(\cdot)))\in L^p(I,\R{})$. Define $\psi=\essinf\{|u(\cdot)|\colon u\in K\}$, where $K:=\{w_1\}-N_F^1(h_2+V(u_2))$. Considering that \[\psi(t)\<d(w_1(t),F(t,h_2(t)+V(u_2)(t)))<d(w_1(t),F(t,h_2(t)+V(u_2)(t)))+\eps\] a.e. in $I$, it follows by \cite[Prop.2.]{brescol} that there is $w_2\in G_p(h_2,u_2)$ such that $|w_1(t)-w_2(t)|<d(w_1(t),F(t,h_2(t)+V(u_2)(t)))+\eps$ for a.a. $t\in I$. Thus, we can estimate
{\allowdisplaybreaks \begin{align*}
|||w_1-w_2|||_p^p&=\int_0^Te^{-2^{2p-1}Mr(t)}|w_1(t)-w_2(t)|^p\,dt\\&\<\int_0^Te^{-2^{2p-1}Mr(t)}\left(d(w_1(t),F(t,h_2(t)+V(u_2)(t)))+\eps\right)^p\,dt\\&\<\int_0^Te^{-2^{2p-1}Mr(t)}\left(d_H(F(t,h_1(t)+V(u_1)(t)),F(t,h_2+V(u_2)(t))+\eps\right)^p\,dt\\&\<\int_0^Te^{-2^{2p-1}Mr(t)}\left(\alpha(t)\left|h_1(t)-h_2(t)+\int_0^tk(t,s)(u_1(s)-u_2(s))\,ds\right|+\eps\right)^p\,dt\\&\<\int_0^Te^{-2^{2p-1}Mr(t)}2^{p-1}\left(\alpha(t)|h_1(t)-h_2(t)|+\eps\right)^p\,dt\\&+\int_0^Te^{-2^{2p-1}Mr(t)}2^{p-1}\alpha(t)^p\left(\int_0^t\!\!\!||k(t,s)||_{\mathcal L}|u_1(s)-u_2(s)|\,ds\right)^p\,dt
\\&\<||h_1-h_2||^p\int_0^T2^{2p-2}e^{-2^{2p-1}Mr(t)}\alpha(t)^p\,dt+2^{2p-2}\eps^p\int_0^Te^{-2^{2p-1}Mr(t)}\,dt\\&+\int_0^Te^{-2^{2p-1}Mr(t)}\alpha(t)^p2^{p-1}||k(t,\cdot)||_q^p\int_0^t|u_1(s)-u_2(s)|^p\,ds\,dt\\&\<\frac{1}{2M}\left(1-e^{-2^{2p-1}Mr(T)}\right)||h_1-h_2||^p+T2^{2p-2}\eps^p\\&+M\int_0^T\int_s^Te^{-2^{2p-1}Mr(t)}\alpha(t)^p2^{p-1}|u_1(s)-u_2(s)|^p\,dt\,ds\\&\<\frac{1}{2M}||h_1-h_2||^p+T2^{2p-2}\eps^p\\&+\frac{1}{2^p}\int_0^T\left(e^{-2^{2p-1}Mr(s)}-e^{-2^{2p-1}Mr(T)}\right)|u_1(s)-u_2(s)|^p\,ds\\&\<\frac{1}{2}||h_1-h_2||^p+\frac{1}{2^p}\int_0^Te^{-2^{2p-1}Mr(s)}|u_1(s)-u_2(s)|^p\,ds\\&-\frac{1}{2^p}e^{-2^{2p-1}Mr(T)}||u_1-u_2||_p^p+T2^{2p-2}\eps^p\\&\<\frac{1}{2}\left(||h_1-h_2||^p+|||u_1-u_2|||_p^p\right)+T2^{2p-2}\eps^p.
\end{align*}}
Since $\eps$ was arbitrarily small and $w_1$ was an arbitrary element of $G_p(h_1,u_1)$ it follows that \[\left(\sup_{w\in G_p(h_1,u_1)}d(w,G_p(h_2,u_2))\right)^p\<\frac{1}{2}\left(||h_1-h_2||^p+|||u_1-u_2|||_p^p\right)\] and consequently \[d_H(G_p(h_1,u_1),G_p(h_2,u_2))\<\left(\frac{1}{2}\left(||h_1-h_2||^p+|||u_1-u_2|||_p^p\right)\right)^{\frac{1}{p}}.\] In particular, \[d_H(G_p(h,u_1),G_p(h,u_2))\<\frac{1}{2^{p^{-1}}}|||u_1-u_2|||_p.\]
Therefore $G_p$ is Hausdorff continuous and contractive with respect to the second variable.
\end{proof}

\section{Main results}
The following result is the announced enhancement of Theorem 6. in \cite{pietkun}, dedicated to the description of geometric structure of the solution set $S^p_{\!F}(h)$.
\begin{theorem}\label{geometry}
Let $p\in[1,\infty)$, while the space $E$ is reflexive for $p\in(1,\infty)$. Assume that $F\colon I\times E\map E$ satisfies $({\bf F})$ and $k\colon\bigtriangleup\to{\mathcal L}(E)$ satisfies $(\Ka_1)$-$(\Ka_4)$. Then the solution set $S^p_{\!F}(h)$ of the integral inclusion \eqref{inclusion} is an $R_\delta$-set in the space $C(I,E)$.
\end{theorem}
\begin{proof}
We adopt the notation and arguments used to justify \cite[Th. 6.]{pietkun} in the context of the assumption $(\E_2)$. Let us recall that there exists a nonempty convex compact set $X\subset C(I,E)$ possessing the following property:
\[h(t)+\int_0^tk(t,s)\,\overline{\co}F(s,X(s))\,ds\subset X(t)\;\;\mbox{on }I.\]
Let $Pr\colon I\times E\map E$ be the time-dependent metric projection, given by:
\[Pr(t,x)=\{y\in X(t)\colon |x-y|=\inf\{|x-z|\colon z\in X(t)\}\}.\]
Observe that $Pr$ is an upper semicontinuous multimap with compact values. Relying on the compactness of the set $X$ define the multimap $\w{F}\colon I\times E\map E$ by the formula:
\[\w{F}(t,x)=\overline{\co}F(t,Pr(t,x)).\]
Note that $\w{F}$ satisfies assumptions $(\F_1)$-$(\F_5)$. Replacement of the map $F$ by the map $\w{F}$ does not change the set of solutions to inclusion \eqref{inclusion}, i.e. $S_{\!F}^p=S^p_{\!\w{F}}$. Using a strictly analogous approach to that applied in the proof of case $(\E_1)$ we can represent the solution set $S_{\!\w{F}}^p$ in the form of a countable intersection of a deacreasing sequence of compact solution sets $S^p_n$ to integral inclusions corresponding to multivalued approximations $F_n$ such that $F_n(t,x)\subset\overline{\co}\,\w{F}(t,B(x,3^{-n+1}))\subset\overline{\co}F(t,X(t))$ for $(t,x)\in I\times E$. Resting on the uniqueness of solutions to integral equations of the form \[x(t)=g(t)+\int_{\tau}^tk(t,s)f_n(s,x(s))\,ds,\] one can show that sets $S_n^p$ are contractible. As a result, we conclude that $S_{\!\w{F}}^p$ is an $R_\delta$-set.\par The choice of measurable-locally Lipschitzian mapping $f_n\colon I\times E\to E$ is based on the existence of strongly measurable selection of the map $\w{F}(\cdot,x)$. Therefore the proof will be completed if we can show that the multimap $\overline{\co}F(\cdot,Pr(\cdot,x))$ possesses a strongly measurable selection.\par Since $Pr(\cdot,x)\colon I\map E$ is a compact valued upper semicontinuous multimap, there exists a sequence $(Pr_n\colon I\map E)_\n$ of compact valued upper semicontinuous step multifunctions such that for every $t\in I$ we have
\[Pr(t,x)\subset Pr_{n+1}(t)\subset Pr_n(t,x)\;\;\;\mbox{and}\;\;\;d_H(Pr_n(t),Pr(t,x))\underset{n\to\infty}{\longrightarrow} 0.\]
Indeed, one can define the map $Pr_n$ in the following way \[Pr_n(t):=\sum_{i=0}^{2^n-1}{\bf 1}_{[t_i^n,t_{i+1}^n]}(t)Pr\left([t_i^n,t_{i+1}^n],x\right),\]
where $t_i^n=i\frac{T}{2^n}$. In view of assumptions $(\F_2)$ and $(\F_4)$ there are functions $w_i^n\in L^p(I,E)$ such that $w_i^n(t)\in F(t,a_i^n)$ a.e. on $I$ for some $a_i^n\in Pr\left([t_i^n,t_{i+1}^n],x\right)$. Let $w_n=\sum_{i=0}^{2^n-1}w_i^n{\bf 1}_{[t_i^n,t_{i+1}^n]}$. The function $w_n$ is also Bochner integrable and more importantly $w_n(t)\in F(t,Pr_n(t))$ for almost all $t\in I$. If $p=1$, we see that for almost all $t\in I$ the set $\{w_n(t)\}_\n$ is contained in the weakly compact set $F\left(t,\bigcup_{t\in I}X(t)\right)$. Thus the sequence $(w_n)_\n$ is relatively weakly compact in $L^1(I,E)$, thanks to \cite[Proposition 11]{ulger}. The Eberlein-\v Smulian theorem implies the relative weak compactness of the sequence $(w_n)_\n$ for $p\in(1,\infty)$. Denote by $w$ the weak limit of some subsequence of $(w_n)_\n$. Applying Theorem \ref{convergence}. we infer that $w(t)\in \overline{\co}F(t,Pr(t,x))$ almost everywhere on $I$. Thus $w$ is the sought strongly measurable selection of $\w{F}(\cdot,x)$, completing the proof.
\end{proof}
\begin{corollary}\label{wniosek3}
Let $E$ be a real Banach space and let $p=1$. Assume that the multimap $F\colon I\times E\map E$ satisfies $({\bf F})$ and the linear operator $A\colon E\to E$ is the infinitesimal generator of a uniformly continuous semigroup of bounded linear operators. Then the set of all mild solutions of the Cauchy problem for the following semilinear differential inclusion 
\begin{equation}\label{234}
\begin{cases}
\dot{x}(t)+Ax(t)\in F(t,x(t))&\mbox{on }I,\\
x(0)=x_0,&
\end{cases}
\end{equation}
is an $R_\delta$-subset of the space $C(I,E)$.
\end{corollary}
\begin{proof}
Denote by $\{U(t)\}_{t\geqslant 0}$ the semigroup generated by the operator $A$. Recall that a continuous function $x\colon I\to E$ is a mild solution to \eqref{234} if there is an integrable $w\in N_F^1(x)$ such that 
\[x(t)=U(t)x_0+\int_0^tU(t-s)w(s)\,ds\;\;\;\text{for }t\in I.\] Uniform continuity of the semigroup $\{U(t)\}_{t\geqslant 0}$ means that the Volterra kernel $k(t,s):=U(t-s)$ satisfies conditions $(\Ka_1)$-$(\Ka_4)$, completing the proof.
\end{proof}
\begin{remark}
It is worth noting at this point that in the particular case of the semilinear differential inclusion \eqref{234} there are well established results regarding the $R_\delta$-structure of the solution set (\cite[Th.5]{bothe},\cite[Cor.5.3.1.]{obu}) and upper semicontinuous dependence on initial conditions (\cite[Cor.5.2.2.]{obu}), which are actually more general than Corollary \ref{wniosek3}., since they are based on the assumption that the operator A generates merely a $C_0$-semigroup.
\end{remark}
In the subsequent propositions we will discuss possible further implications of the structure result proven above. Following the established terminology of \cite{pazy}, let us recall that we are dealing with the parabolic initial value problem for an evolution inclusion 
\begin{equation}\label{evolution}
\begin{cases}
\dot{x}(t)+A(t)x(t)\in F(t,x(t))&\mbox{on }I,\\
x(0)=x_0,&
\end{cases}
\end{equation}
if the mentioned below assumptions are met:
\begin{itemize}
\item[$(P_1)$] the domain $D(A(t))=D$ of $A(t)$, $t\in I$ is dense in $E$ and independent of $t$, 
\item[$(P_2)$] for $t\in I$, the resolvent $R(\lambdaup:A(t))$ of $A(t)$ exists for all $\lambdaup$ with $\re\lambdaup\<0$ and there is a constant $M$ such that \[||R(\lambdaup:A(t))||_{\mathcal L}\<\frac{M}{|\lambdaup|+1}\;\;\;\mbox{for }\re\lambdaup\<0, t\in I,\]
\item[$(P_3)$] there exists constants $L$ and $0<\alpha\<1$ such that \[||(A(t)-A(s))A(\tau)^{-1}||_{\mathcal L}\<L|t-s|^\alpha\;\;\;\mbox{for }s,t,\tau\in I.\]
\end{itemize}
\begin{theorem}
Let $E$ be a real Banach space and let $(P_1)$-$(P_3)$ be satisfied. Moreover, assume that
\begin{itemize}
\item[(i)] the domain $D(A(t))=E$ for all $t\in I$ and independent of $t$,
\item[(ii)] the transformation $I\ni t\mapsto A(t)^{-1}x\in E$ is continuous for every $x\in E$,
\item[(iii)] the unique evolution system $U(t,s)$ generated by the family $\{A(t)\}_{t\in I}$ is continuous with respect to the second variable in the uniform operator topology, i.e. $U(t,\cdot)\colon[0,t]\to{\mathcal L}(E)$ is continuous for every $t\in I$.  
\end{itemize}
If the multivalued perturbation $F\colon I\times E\map E$ satisfies conditions $({\bf F})$ with $p=1$, then the set $S_{\!F}(x_0)\subset C(I,E)$ of all mild solutions of the parabolic initial value problem \eqref{evolution} is an $R_\delta$-set.
\end{theorem}
\begin{proof}
Remind that $x\in S_{\!F}(x_0)$ iff \[x(t)\in U(t,0)x_0+\int_0^t U(t,s)F(s,x(s))\,ds\;\;\mbox{ on }I.\] Observe that all the hypotheses of Theorem \ref{geometry}. are fulfilled, except the condition $(\Ka_4)$. In particular, the evolution system $U\colon\bigtriangleup\to{\mathcal L}(E)$ satisfies properties $(\Ka_1)$-$(\Ka_3)$. The property $(\Ka_1)$ results from \cite[Th.6.1.]{pazy} (item $(E_2)^+$). Condition $(\Ka_2)$ is equivalent to (iii) and $(\Ka_3)$ follows from the very definition of the evolution system.\par Lemma 2. in \cite{pietkun} is crucial for the proof of the Theorem 6. in \cite{pietkun} and, consequently, for an enhanced version given in Theorem \ref{geometry}. What is important, the condition $(\Ka_4)$ is used exclusively to justify this particular lemma. From this it follows that the present proof will be completed if we show that the Volterra integral operator $V\colon L^1(I,E)\to C(I,E)$, given by $V(w)(t):=\int_0^tU(t,s)w(s)\,ds$, is a monomorphism. \par Take $w\in L^1(I,E)$. By virtue of (ii) the mapping $[0,t]\ni s\mapsto A(s)^{-1}w(s)\in E$ is strongly measurable. Note that \[|U(t,s)A(s)^{-1}w(s)|\<||U(t,s)||_{\mathcal L}||A(s)^{-1}||_{\mathcal L}|w(s)|\<CM|w(s)|,\] in view of property $(P_2)$ and $(E_1)'$ in \cite[Th.6.1.]{pazy}. Therefore, the following modified integral operator $\w{V}\colon L^1(I,E)\to C(I,E)$, \[\w{V}(w)(t):=\int_0^tU(t,s)A(s)^{-1}w(s)\,ds,\] is properly defined. \par Let $\w{V}(w_1)=\w{V}(w_2)$ and $w=w_1-w_2$. Basically, we will reproduce the approach used in the proof of Lemma 2. in \cite{pietkun}. Since $\frac{d}{dt}\w{V}(w)(t)=0$, we have
\begin{equation}\label{zero}
\lim_{n\to\infty}\left(\int\limits_0^{t-\frac{1}{n}}\frac{U\left(t-\frac{1}{n},s\right)A(s)^{-1}-U(t,s)A(s)^{-1}}{-\frac{1}{n}}w(s)\,ds+n\int\limits_{t-\frac{1}{n}}^tU(t,s)A(s)^{-1}w(s)\,ds\right)=0
\end{equation}
for every $t\in I$. If \[f_n(s)=\frac{U\left(t-\frac{1}{n},s\right)A(s)^{-1}-U(t,s)A(s)^{-1}}{-\frac{1}{n}}w(s){\bf 1}_{\left[0,\,t-\frac{1}{n}\right]}(s),\] then $f_n\in L^1([0,t],E)$ and $f_n(s)\xrightarrow[n\to\infty]{}\frac{\partial}{\partial t}U(t,s)A(s)^{-1}w(s)$ for $s\in[0,t)$. Using the findings of \cite[Th.6.1]{pazy} we can estimate
\begin{align*}
|f_n(s)|&\<\!\left\Arrowvert\frac{U\left(t-\frac{1}{n},s\right)A(s)^{-1}-U(t,s)A(s)^{-1}}{-\frac{1}{n}}\right\Arrowvert_{\mathcal L}\!\!\!|w(s)|\<\!\!\!\sup_{\xi\in\left[t-\frac{1}{n},t\right]}\left\Arrowvert\frac{\partial}{\partial t}U\left(\xi,s\right)A(s)^{-1}\right\Arrowvert_{\mathcal L}\!\!\!|w(s)|\\&=\sup_{\xi\in\left[t-\frac{1}{n},t\right]}\left\Arrowvert A(\xi)U\left(\xi,s\right)A(s)^{-1}\right\Arrowvert_{\mathcal L}|w(s)|\<C|w(s)|
\end{align*} 
for all $\n$ and $s\in\left[0,t-\frac{1}{n}\right)$. Consequently, there is a convergence 
\[\int\limits_0^{t-\frac{1}{n}}\frac{U\left(t-\frac{1}{n},s\right)A(s)^{-1}-U(t,s)A(s)^{-1}}{-\frac{1}{n}}w(s)\,ds\xrightarrow[n\to\infty]{}\int_0^t\frac{\partial}{\partial t}U(t,s)A(s)^{-1}w(s)\,ds\] for every $t\in I$. On the other side, we can evaluate the second term of \eqref{zero} in the following way:
\begin{align*}
n\!\int\limits_{t-\frac{1}{n}}^t|U(t,s)A(s)^{-1}w(s)-U(t,t)A(t)^{-1}w(t)|\,ds&\<||U(t,\xi(n))-U(t,t)||_{\mathcal L}n\int\limits_{t-\frac{1}{n}}^t|A(s)^{-1}w(s)|\,ds\\&+||U(t,t)||_{\mathcal L}n\!\int\limits_{t-\frac{1}{n}}^t|A(s)^{-1}w(s)-A(t)^{-1}w(t)|\,ds
\end{align*}
for some $\xi(n)\in\left[t-\frac{1}{n},t\right]$. Using the characteristic of Lebesgue points and the continuity of $U(t,\cdot)$ we obtain
\[\lim_{n\to\infty}n\int\limits_{t-\frac{1}{n}}^tU(t,s)A(s)^{-1}w(s)\,ds=U(t,t)A(t)^{-1}w(t)\] for a.a. $t\in I$. Applying \eqref{zero} one gets
\begin{equation}\label{zero1}
A(t)^{-1}w(t)=-\int_0^t\frac{\partial}{\partial t}U(t,s)A(s)^{-1}w(s)\,ds
\end{equation}
for a.a. $t\in I$.\par Observe that $\{A(t)\}_{t\in I}\subset{\mathcal L}(E)$, in view of the assumptions $(P_2)$ and (i). We claim that the family $\{A(t)\}_{t\in I}$ is strongly continuous, i.e. the map $I\ni t\mapsto A(t)x\in E$ is continuous for every fixed $x\in E$. Considering $(P_3)$, we have
\[||id-(id-A(t)(A(t)^{-1}-A(t+h)^{-1}))||_{\mathcal L}=||id-A(t)A(t+h)^{-1}||_{\mathcal L}\<L|h|^\alpha\] and consequently \[||(id-A(t)(A(t)^{-1}-A(t+h)^{-1}))^{-1}||_{\mathcal L}\<\frac{1}{1-||id-A(t)A(t+h)^{-1}||_{\mathcal L}}\<\frac{1}{1-L|h|^\alpha}.\] Thus \[||A(t+h)||_{\mathcal L}=||(id-A(t)(A(t)^{-1}-A(t+h)^{-1}))^{-1}A(t)||_{\mathcal L}\<\frac{||A(t)||_{\mathcal L}}{1-L|h|^\alpha}\<2||A(t)||_{\mathcal L}\] for sufficiently small $h$. Fix $x\in E$. Thanks to the above estimation we get: 
\begin{align*}
|A(t+h)x-A(t)x|&=|A(t+h)(A(t)^{-1}-A(t+h)^{-1})A(t)x|\\&\<||A(t+h)||_{\mathcal L}|(A(t)^{-1}-A(t+h)^{-1})A(t)x|\\&\<2||A(t)||_{\mathcal L}|x-A(t+h)^{-1}(A(t)x)|.
\end{align*}
for sufficiently small $h$. From (ii) it follows that the last quantity tends to zero as $h\to 0$. The uniform boundedness principle and the strong continuity of the family $\{A(t)\}_{t\in I}$ implies $R:=\sup_{t\in I}||A(t)||_{\mathcal L}<\infty$.\par As a consequence of \eqref{zero1} and \cite[Th.6.1.]{pazy} we obtain the following estimation
\[|w(t)|\<||A(t)||_{\mathcal L}\int_0^t\left\Arrowvert\frac{\partial}{\partial t}U(t,s)A(s)^{-1}\right\Arrowvert_{\mathcal L}|w(s)|\,ds\<R\int_0^tC|w(s)|\,ds\] for a.a. $t\in I$. In view of the Gronwall inequality, it is clear that $|w(t)|=0$ for a.a. $t\in I$. Thus $w_1=w_2$ in $L^1(I,E)$ and the operator $\w{V}$ is injective.\par Let $V(w_1)=V(w_2)$. Taking into account previously demonstrated properties of $\{A(t)\}_{t\in I}$ we see that $A(\cdot)w_1,A(\cdot)w_2\in L^1(I,E)$. Now it suffices to note that $\w{V}(A(\cdot)w_1)=\w{V}(A(\cdot)w_2)$. Hence $A(t)w_1(t)=A(t)w_2(t)$ for almost all $t\in I$. Eventually, $w_1(t)=A(t)^{-1}A(t)w_1(t)=A(t)^{-1}A(t)w_2(t)=w_2(t)$ a.e. on $I$, completing the proof.
\end{proof}
\par Let us recall that a set-valued map $F\colon U\subset E\map E$ is strongly upper semicontinuous if, for every sequence $(x_n)_\n$ in $U$ such that $x_n\rightharpoonup x_0$ and any sequence $(y_n)_\n$ satisfying $y_n\in F(x_n)$ for all $\n$, there is a subsequence $(y_{k_n})_\n$ such that $y_{k_n}\to y_0\in F(x_0)$. Providing another example of the application of Theorem \ref{geometry}., we will refer to the following fixed point result:
\begin{lemma}\label{fixed}
Let $E$ be a reflexive Banach space and let $\varphi\colon\overline{B}(0,R)\subset E\map\overline{B}\left(0,\frac{R}{2}\right)$ be a decomposable strongly upper semicontinuous map. Assume there is $S\in{\mathcal L}(E)$ with $||S||\<\frac{1}{2}$. Then the multimap $\Phi:=S+\varphi\colon\overline{B}(0,R)\map E$ has a fixed point.
\end{lemma}
\begin{proof}
Observe that $S$ is a nonexpansive $J$-map and $\Phi$ satisfies Yamamuro's condition, i.e. \[\exists\;z\in B(0,R)\;\;\forall\;x\in\partial B(0,R)\;\;\;\lambdaup(x-z)\in\Phi(x)-\{z\}\Rightarrow\lambdaup\<1.\] Indeed, set $z:=0$ and note that for $x\in\partial B(0,R)$ we have \[|\lambdaup|=\frac{|\lambdaup x|}{|x|}\<\frac{|\lambdaup x-Sx|+|Sx|}{|x|}\<\left(\frac{R}{2}+\frac{R}{2}\right)\frac{1}{R}=1.\] Proceed along the proof of \cite[Cor.11.]{bader} to convince oneself that $\fix(\Phi)\neq\varnothing$.
\end{proof}
\begin{remark}
Of course, the condition of Opial used to justify \cite[Cor.11.]{bader} is superfluous in our case, when $S$ is a single-valued bounded endomorphism. 
\end{remark}
\begin{theorem}
Let $E$ be a reflexive Banach space and $p\in[1,\infty)$. Assume that conditions $({\bf F})$ and $(\Ka_1)$-$(\Ka_4)$ are satisfied. If $F(t,\cdot)$ is strongly upper semicontinuous for a.a. $t\in I$, then there exists a continuous function $h\colon I\to E$, for which the integral inclusion \eqref{inclusion} possesses a $T$-periodic solution.
\end{theorem}
\begin{proof}
At the beginning note that we can assume, w.l.o.g., that the map $F$ is integrably bounded, i.e. there exists $\mu\in L^p(I,\R{})$ such that $||F(t,x)||^+\<\mu(t)$ for every $x\in E$ and for a.a. $t\in I$ (see commentary at the beginning of Section 3. in \cite{pietkun}).\par Let $\{U(t)\}_{t\in I}\subset{\mathcal L}(E)$ be a family of operators possessing the following properties: 
\begin{itemize}
\item[(i)] $U(0)=id$,
\item[(ii)] $\{U(t)\}_{t\in I}$ is strongly continuous, i.e. the map $I\ni t\mapsto U(t)x\in E$ is continuous for every $x\in E$,
\item[(iii)] $||U(T)||_{\mathcal L}\<\frac{1}{2}$.
\end{itemize}
Let us define the multimap $\,\varphi\colon E\map E$ in the following way: $\varphi(x):=V\left({\mathcal S}_F^p(U(\cdot)x)\right)(T)$. Clearly, \[|V(w)(T)|\<R:=||k(T,\cdot)||_q||\mu||_p\] for every $w\in{\mathcal S}_F^p(U(\cdot)x)$ and for all $x\in E$. Therefore $\varphi(\overline{B}(0,2R)\subset\overline{B}(0,R)$. We claim that the Poincar\'e-like operator $P\colon\overline{B}(0,2R)\map E$, given by $P(x):=U(T)x+\varphi(x)$, possesses a fixed point $x_0\in P(x_0)$. This fixed point is associated with a $T$-periodic solution to the problem \eqref{inclusion}, with the inhomogeneity $h\in C(I,E)$ given by $h(t)=U(t)x_0$. In view of Lemma \ref{fixed}. the present proof will be completed if we can only show that $\varphi$ is strongly usc decomposable map.\par Since the geometry of $R_\delta$-type sets is invariant under translation, the multi-valued map $E\ni x\mapsto S_{\!F}^p(U(\cdot)x)-U(\cdot)x\subset C(I,E)$ possesses values of $R_\delta$-type (by virtue of Theorem \ref{geometry}.). Note that $\varphi$ is a composition of this operator and the evaluation at time $T$.\par Take $(x_n)_\n$ in $\overline{B}(0,2R)$ such that $x_n\rightharpoonup x_0$. Let $\w{y}_n\in\varphi(x_n)$ for $\n$. Clearly, $\w{y}_n=y_n(T)$, where $y_n=V(w_n)$ for some $w_n\in N_F^p(z_n)$ such that $z_n=U(\cdot)x_n+V(w_n)$. The image $F(t,\{z_n(t)\}_\n)$ is relatively compact for a.a. $t\in I$ due to the following bound \[\sup_\n|z_n(t)|\<\sup_{t\in I}||U(t)||_{\mathcal L}\sup_\n|x_n|+\sup_{t\in I}||k(t,\cdot)||_q||\mu||_p\] and in view of the strong upper semicontinuity of $F(t,\cdot )$. Thus, on the one hand we obtain the relative weak compactness of $(w_n)_\n$ in $L^1(I,E)$ (through the use of \cite[Pro.11.]{ulger}). On the other, through the application of \cite[Cor.3.1]{heinz} (compare \cite[Th.3.]{pietkun}) in the following estimate \[\chi(\{y_n(t)\}_\n)\<\chi\left(\left\{\int_0^tk(t,s)w_n(s)\,ds\right\}_\n\right)\<2\int_0^t||k(t,s)||_{\mathcal L}\chi(F(s,\{z_n(s)\}_\n))\,ds\] we get the relative compactness of $\{y_n(t)\}_\n$ for every $t\in I$. It is easy to notice that the family $\{y_n\}_\n$ is equicontinuous. It follows from \[\sup_\n|y_n(t)-y_n(\tau)|\<||k(t,\cdot)-k(\tau,\cdot)||_q||\mu||_p+\sup_{t\in I}||k(t,\cdot)||_q\left(\int_{\tau}^t\mu(s)^p\,ds\right)^\frac{1}{p}.\] Therefore, passing to a subsequence if necessary we have the convergence $y_n\rightrightarrows y_0$. \par Denote by $w_0\in L^p(I,E)$ the weak limit of some subsequence of $(w_n)_\n$. Put $z_0(t):=U(t)x_0+V(w_0)(t)$ for $t\in I$. It suffices to show that $w_0\in N_F^p(z_0)$ to complete the proof. Indeed, for then $y_n=V(w_n)\rightrightarrows V(w_0)\in V({\mathcal S}_F^p(U(\cdot)x_0))$. This means that the operator $x\mapsto V\left({\mathcal S}_F^p(U(\cdot)x)\right)$ is a $J$-mapping and eventually $\varphi$ is decomposable. Moreover, setting $\w{y}_0:=V(w_0)(T)$ we obtain in the result that $\w{y}_n\to\w{y}_0\in\varphi(x_0)$.\par In order to demonstrate that $w_0(t)\in F(t,z_0(t))$ a.e. on $I$, it is enough to duplicate the routine argumentation involving Mazur's lemma. We omit the details of the proof and instead of this evoke some version of convergence theorem given in Lemma 1. in \cite{cichon}. It is obvious that for every sequence $(x_n)_\n$ in $E$ converging weakly to some $x_0$ holds the inequality: $\limsup_{n\to\infty}\sigma(e^*,F(t,x_n))\<\sigma(e^*,F(t,x_0))$ for every $e^*\in E^*$ and for a.a. $t\in I$. This indicates that the multimap $F(t,\cdot)$ is weakly sequentially upper hemicontinuous (within the meaning of Definition 1. in \cite{cichon}) a.e. on $I$. Now, observe that $z_n(t)=U(t)x_n+V(w_n)(t)\rightharpoonup z_0(t)$ for $t\in I$ (in fact, $z_n\rightharpoonup z_0$ in $C(I,E)$, but this knowledge is redundant). Therefore, all the hypotheses of \cite[Lem.1.]{cichon} are met (confront Remark 1. in \cite{pietkun}) and $w_0(t)\in F(t,z_0(t))$ a.e. on $I$.
\end{proof}
Below we formulate the thesis about the existence of a continuous selection of the solution set map $S_{\!F}^p$, which is part of a broader research trend initiated in \cite{cell, colo}.
\begin{theorem}
Let $p\in[1,\infty)$ and $E$ be a separable Banach space. Assume that $F\colon I\times E\map E$ satisfies $({\bf H})$ and $k\colon\bigtriangleup\to{\mathcal L}(E)$ satisfies $(\Ka_5)$-$(\Ka_6)$. Then the solution set map $S^p_{\!F}$ possesses a continuous selection.
\end{theorem}
\begin{proof}
Fix $\eps>0$, set $\eps_n:=\frac{n+1}{n+2}\eps$, $M:=\sup\limits_{t\in I}||k(t,\cdot)||_q$ and $m(t):=\int_0^t\alpha(s)^p\,ds$. Let $\gamma\colon C(I,E)\to L^1(I,\R{})$ be a continuous function, given by \[\gamma(h)(t):=2^p\max\left\{\beta(t)^p,\alpha(t)^p|ev(t,h)|\right\},\] where $ev(t,h)=h(t)$ is the evaluation. Define $\beta_n\colon C(I,E)\to L^1(I,\R{})$ in the following way 
\[\beta_n(h)(t)=M^{np}\left(\int_0^t\gamma(h)(s)\frac{(m(t)-m(s))^{n-1}}{(n-1)!}\,ds+T\eps_n\frac{m(t)^{n-1}}{(n-1)!}\right).\]
\par It follows from $(\h_3)$ that $F(t,\cdot)$ is Hausdorff continuous. Therefore, $F$ is $\Sigma\otimes{\mathcal B}(E)$-measurable in view of \cite[Th.3.3]{papa}. Theorem 1. in \cite{zyg} implies that the set-valued map $F(\cdot,ev(\cdot,h))$ is $\Sigma$-measurable. Thanks to $(\h_3)$ the multimap $F(t,ev(t,\cdot))$ is Hausdorff continuous. Hence, applying again Theorem 3.3. in \cite{papa} we see that $F(\cdot,ev(\cdot,\cdot))$ is $\Sigma\otimes{\mathcal B}(C(I,E))$-measurable. Observe that the multivalued map $F_0\colon C(I,E)\map L^1(I,E)$ such that \[F_0(h):=\left\{w\in L^1(I,E)\colon w(t)\in F(t,ev(t,h))\mbox{ a.e. in }I\right\}\] is lower semicontinuous and possesses nonempty closed decomposable values. This is consequence of Proposition 2.1. in \cite{colo}, since $d(0,F(t,ev(t,h)))\<\beta(t)+\alpha(t)|ev(t,h)|^{\frac{1}{p}}$ for a.a. $t\in I$. In fact, $F_0(h)\subset L^p(I,E)$, because $\inf\{|z|\colon z\in F(t,ev(t,h))\}\in L^p(I,\R{})$.\par Let $H_0\colon C(I,E)\map L^1(I,E)$ be a multimap described by the formula \[H_0(h):=\left\{w\in F_0(h)\colon |w(t)|<\left(\gamma(h)(t)+\eps_0\right)^{\frac{1}{p}}\mbox{ for a.a. }t\in I\right\}.\] We claim that it has nonempty values. Indeed, define the nonempty closed decomposable set $K:=N^1_F(ev(\cdot,h))$. Let $\psi\colon I\to\R{}_+$ be defined by $\psi=\essinf\{|w(\cdot)|\colon w\in K\}$. Using Castaing Representation we may write $F(t,ev(t,h))=\overline{\{g_k(t)\}}_\K$. Thus, for every $\K$, $\psi(t)\<|g_k(t)|$ a.e. in $I$. Obviously, there is a subset $J\subset I$ of full measure such that for every $t\in J$, $\psi(t)\<\inf_\K|g_k(t)|$. Consequently, 
\begin{align*}
\psi(t)\<d\left(0,\{g_k(t)\}_\K\right)&=d\left(0,\overline{\{g_k(t)\}}_\K\right)=d(0,F(t,ev(t,h)))\\&\<\left(2^p\max\{\beta(t)^p,\alpha(t)^p|ev(t,h)|\}\right)^{\frac{1}{p}}<\left(\gamma(h)(t)+\eps_0\right)^{\frac{1}{p}}
\end{align*}
a.e. in $I$. Applying Proposition 2. in \cite{brescol} we get $w\in L^p(I,E)$ such that $w(t)\in F(t,ev(t,h))$ and $|w(t)|<\left(\gamma(h)(t)+\eps_0\right)^{\frac{1}{p}}$ a.e. in $I$.  
\par Proposition 4. and Theorem 3. in \cite{brescol} imply the existence of continuous selection $g_0\colon C(I,E)\to L^p(I,E)$ of $H_0$. Define $f_0(t,h):=g_0(h)(t)$ and \[x_1(t,h):=h(t)+\int_0^tk(t,s)f_0(s,h)\,ds,\;\;t\in I.\] Observe that
\begin{align*}
|x_1(t,h)-ev(t,h)|^p&\<\left(\int_0^t||k(t,s)||_{\mathcal L}|f_0(s,h)|\,ds\right)^p\<||k(t,\cdot)||_q^p\int_0^t(\gamma(h)(s)+\eps_0)\,ds\\&<M^p\int_0^t(\gamma(h)(s)+\eps_1)\,ds=\beta_1(h)(t)
\end{align*}
for each $h\in C(I,E)$ and $t\in(0,T]$.\par We maintain that there exist Cauchy sequences $(f_n(\cdot,h))_\n\subset L^p(I,E)$ and $(x_n(\cdot,h))_\n\subset C(I,E)$, such that for all $\n$ the following properties hold:
\begin{itemize}
\item[(i)] $C(I,E)\ni h\mapsto f_n(\cdot,h)\in L^p(I,E)$ is continuous,
\item[(ii)] $f_n(t,h)\in F(t,x_n(t,h))$ for every $h\in C(I,E)$ and a.e. $t\in I$,
\item[(iii)] $|f_n(t,h)-f_{n-1}(t,h)|\<\alpha(t)(\beta_n(h)(t))^{\frac{1}{p}}$ for a.a. $t\in I$,
\item[(iv)] $x_{n+1}(t,h)=h(t)+\int_0^tk(t,s)f_n(s,h)\,ds$ for $t\in I$.
\end{itemize}
\par Suppose that $(f_k(\cdot,h))_{k=1}^n$ and $(x_k(\cdot,h))_{k=1}^n$ are constructed. Define $x_{n+1}(\cdot,h)\colon I\to E$ in accordance with (iv). By virtue of (i) it follows that \[x_n(\cdot,h_0)=h_0+V(f_{n-1}(\cdot,h_0))=\lim_{h\rightrightarrows h_0}h+V(f_{n-1}(\cdot,h))=\lim_{h\rightrightarrows h_0}x_n(\cdot,h),\]
i.e. the mappings $C(I,E)\ni h\mapsto x_n(\cdot,h)\in C(I,E)$ are also continuous. Following calculations contained in \cite{aubin} (formula (14), page 122) we see that
{\allowdisplaybreaks \begin{equation}\label{intbeta}
\begin{split}
&M^p\int_0^t\alpha(s)^p\beta_n(h)(s)\,ds\\&=M^{(n+1)p}\left(\int_0^t\int_0^s\alpha(s)^p\gamma(h)(\tau)\frac{(m(s)-m(\tau))^{n-1}}{(n-1)!}\,d\tau ds+\int_0^tT\eps_n\alpha(s)^p\frac{m(s)^{n-1}}{(n-1)!}\,ds\right)\\&=M^{(n+1)p}\left(\int_0^t\gamma(h)(\tau)\int_\tau^t\frac{1}{n!}\frac{d}{ds}(m(s)-m(\tau))^n\,ds d\tau+T\eps_n\int_0^t\frac{1}{n!}\frac{d}{ds}m(s)^n\,ds\right)\\&<M^{(n+1)p}\left(\int_0^t\gamma(h)(s)\frac{(m(t)-m(s))^n}{n!}\,ds+T\eps_{n+1}\frac{m(t)^n}{n!}\right)=\beta_{n+1}(h)(t).
\end{split}
\end{equation}}
Hence 
\begin{align*}
|x_{n+1}(t,h)-x_n(t,h)|^p&\<M^p\int_0^t|f_n(s,h)-f_{n-1}(s,h)|^p\,ds\<M^p\int_0^t\alpha(s)^p\beta_n(h)(s)\,ds\\&<\beta_{n+1}(h)(t)
\end{align*}
for $t\in(0,T]$. Let $F_{n+1}\colon C(I,E)\map L^1(I,E)$ be such that \[F_{n+1}(h):=\{w\in L^1(I,E)\colon w(t)\in F(t,x_{n+1}(t,h))\mbox{ a.e. in }I\}.\] Then $F_{n+1}$ is lower semicontinuous with nonempty closed decomposable values, by exactly the same arguments as the multimap $F_0$. Define $H_{n+1}\colon C(I,E)\map L^1(I,E)$ in the following way \[H_{n+1}(h):=\left\{w\in F_{n+1}(h)\colon |w(t)-f_n(t,h)|<\alpha(t)(\beta_{n+1}(h)(t))^{\frac{1}{p}}\mbox{ a.e. in }I\right\}.\] Notice that 
\begin{align*}
d(f_n(t,h),F(t,x_{n+1}(t,h)))&\<d(f_n(t,h),F(t,x_n(t,h)))+d_H(F(t,x_n(t,h)),F(t,x_{n+1}(t,h)))\\&\<\alpha(t)|x_n(t,h)-x_{n+1}(t,h)|<\alpha(t)(\beta_{n+1}(h)(t))^{\frac{1}{p}}
\end{align*}
a.e. in $I$. Imitating previous arguments one can show that $H_{n+1}(h)\neq\varnothing$ (applying these to the nonempty closed decomposable set $K:=\{f_n(\cdot,h)\}-N_F^1(x_{n+1}(\cdot,h))$). Moreover, $H_{n+1}(h)\subset L^p(I,E)$. Let $g_{n+1}\colon C(I,E)\to L^p(I,E)$ be a continuous selection of the multimap $H_{n+1}$. Now, if we put $f_{n+1}(t,h):=g_{n+1}(h)(t)$ for $(t,h)\in I\times C(I,E)$, then it is obvious that the function $f_{n+1}$ possesses properties (i)-(iii). \par Making use of inequality \eqref{intbeta} we can easily estimate
\begin{align*}
||x_{n+1}(\cdot,h)-x_n(\cdot,h)||&\<\sup_{t\in I}||k(t,\cdot)||_q||f_n(\cdot,h)-f_{n-1}(\cdot,h)||_p\<M\left(\int_0^T\alpha(t)^p\beta_n(h)(t)\,dt\right)^{\frac{1}{p}}\\&\<\left(M^{(n+1)p}\left(\int_0^T\gamma(h)(t)\frac{(m(T)-m(t))^n}{n!}\,dt+T\eps_{n+1}\frac{m(T)^n}{n!}\right)\right)^{\frac{1}{p}}\\&\<\frac{M^{n+1}||\alpha||_p^n}{(n!)^{\frac{1}{p}}}\left(||\gamma(h)||_1+T\eps\right)^{\frac{1}{p}}.
\end{align*}
Since the series $\sum_{n=1}^\infty\frac{M^{n+1}||\alpha||_p^n}{(n!)^{\frac{1}{p}}}$ is convergent, the sequences $(x_n(\cdot,h))_\n$ and $(f_n(\cdot,h))_\n$ are fundamental and ultimately converge to some $x(\cdot,h)\in C(I,E)$ and $f(\cdot,h)\in L^p(I,E)$ respectively. Moreover, sequences $(h\mapsto x_n(\cdot,h))_\n$ and $(h\mapsto f_n(\cdot,h))_\n$ tend locally uniformly to the limit functions $h\mapsto x(\cdot,h)$ and $h\mapsto f(\cdot,h)$, because $||\gamma(\cdot)||_1$ is locally bounded. Therefore mappings $h\mapsto x(\cdot,h)$ and $h\mapsto f(\cdot,h)$ are also continuous.\par Taking into account the estimation
\begin{align*}
d(f(t,h),F(t,x(t,h)))&\<d(f_n(t,h),F(t,x_n(t,h)))+d_H(F(t,x(t,h)),F(t,x_n(t,h)))\\&+|f_n(t,h)-f(t,h)|\<\alpha(t)|x_n(t,h)-x(t,h)|+|f_n(t,h)-f(t,h)|
\end{align*}
a.e. in $I$ and almost everywhere convergence of $(f_n(\cdot,h))_\n$ along some subsequence, we infer that $f(t,h)\in F(t,x(t,h))$ for each $h\in C(I,E)$ and for a.a. $t\in I$. Passing to the limit in (iv), uniformly with respect to $t$, we obtain \[x(\cdot,h)=\lim_{n\to\infty}x_n(\cdot,h)=\lim_{n\to\infty}h+V(f_n(\cdot,h))=h+V(f(\cdot,h))\in h+V\circ N_F^p(x(\cdot,h)),\] i.e. $x(\cdot,h)\in S_{\!F}^p(h)$, completing the proof.
\end{proof}
It is obvious that under the conditions of Theorem \ref{geometry}. the sum $\bigcup_{h\in M}S_{\!F}^p(h)$ is connected, provided $M$ is connected. We can say more about the image $S_{\!F}^p(M)$ if $M\subset C(I,E)$ is convex, namely:
\begin{theorem}
Let $p\in[1,\infty)$, while the space $E$ is reflexive for $p\in(1,\infty)$. Assume that conditions $({\bf F})$ and $(\Ka_1)$-$(\Ka_4)$ are fulfilled. Then the image $S^p_{\!F}(M)$ of a compact convex set $M$ by the solution set map $S^p_{\!F}$ is compact acyclic.
\end{theorem}
\begin{proof}
The image $S^p_{\!F}(M)$ is obviously compact since the solution set map $S^p_{\!F}$ is compact valued upper semicontinuous multifunction (\cite[Prop. 4.]{pietkun}). In view of Theorem \ref{geometry}. it is an acyclic multimap. Therefore the projection $\pi_1^{S_{\!F}^p}\colon\Gamma(S_{\!F}^p)\Rightarrow M$ of the graph $\Gamma(S_{\!F}^p)=\{(h,x)\in M\times C(I,E)\colon x\in S_{\!F}^p(h)\}$ onto the set $M$ is a Vietoris map. Applying Vietoris-Begle theorem (\cite[Th. 8.7.]{gorn}) we see that the induced map $(\pi_1^{S_{\!F}^p})_*\colon H_*(\Gamma(S_{\!F}^p))\longrightarrow^{\hspace{-0.35cm}\approx}\hspace{0.2cm}H_*(M)$ is an isomorphism.\par Observe that the solution set map $S_{\!F}^p$ has convex fibers. Indeed, take $h_1,h_2\in(S_{\!F}^p)^{-1}(\{x\})$ and $\lambda\in[0,1]$. Then there are $w_1,w_2\in N_F^p(x)$ such that $h_1+V(w_1)=x=h_2+V(w_2)$. Thus $x=\lambda x+(1-\lambda)x=\lambda h_1+(1-\lambda)h_2+V(\lambda w_1+(1-\lambda)w_2)\in\lambda h_1+(1-\lambda)h_2+V\circ N_F^p(x)$, since $N_F^p$ is convex-valued. Therefore $\lambda h_1+(1-\lambda)h_2\in(S_{\!F}^p)^{-1}(\{x\})$.\par Define multimap $G\colon S_{\!F}^p(M)\map M$ by $G(x):=\{h\in M\colon x\in S_{\!F}^p(h)\}$. Clearly $G$ is a compact valued upper semicontinuous map. Moreover, it has convex values, because $G(x)=M\cap(S_{\!F}^p)^{-1}(\{x\})$. Since the projection $\pi_1^G\colon\Gamma(G)\Rightarrow S_{\!F}^p(M)$ of the graph of $G$ onto the domain $S_{\!F}^p(M)$ is a Vietoris map, the induced map $(\pi_1^G)_*\colon H_*(\Gamma(G))\longrightarrow^{\hspace{-0.35cm}\approx}\hspace{0.2cm}H_*(S_{\!F}^p(M))$ gives an isomorphism.\par It is easy to indicate the homeomorphism between the graphs $\Gamma(S_{\!F}^p)$ and $\Gamma(G)$. Indeed, let $f\colon\Gamma(S_{\!F}^p)\to\Gamma(G)$ be the product mapping $f=\pi_2^{S_{\!F}^p}\times\pi_1^{S_{\!F}^p}$, i.e. $f(h,x)=(\pi_2^{S_{\!F}^p}(h,x),\pi_1^{S_{\!F}^p}(h,x))=(x,h)$. Then $\pi_2^G\times\pi_1^G\colon\Gamma(G)\to\Gamma(S_{\!F}^p)$ provides the continuous inverse of $f$. Consequently, $f_*\colon H_*(\Gamma(S_{\!F}^p))\longrightarrow^{\hspace{-0.35cm}\approx}\hspace{0.2cm} H_*(\Gamma(G))$ gives an isomorphism of \v{C}ech homology spaces.\par Combining above findings, we see that the homologies $H_*(S_{\!F}^p(M))$ and $H_*(M)$ are isomorphic. Therefore the set $S^p_{\!F}(M)$ must be acyclic.
\end{proof}

Let us pass on to description of the properties of the selection set map ${\mathcal S}^p_{\!F}\colon C(I,E)\map L^p(I,E)$. We start with a fairly obvious:
\begin{proposition}
Let $p\in[1,\infty)$, while the space $E$ is reflexive for $p\in(1,\infty)$. Assume that conditions $({\bf F})$ and $(\Ka_5)$-$(\Ka_6)$ are fulfilled. Then the multivalued map ${\mathcal S}^p_{\!F}$ is weakly upper semicontinuous and has weakly compact values.
\end{proposition}
\begin{proof}
The claim is a straightforward consequence of upper semicontinuity of the solution set map $S^p_{\!F}$ and weak upper semicontinuity of the Nemytskii operator $N_F^p$ (Proposition 4. and Proposition 1. in \cite{pietkun} respectively).
\end{proof}
\begin{theorem}\label{ar}
Let $E$ be a separable Banach space. Assume that conditions $({\bf H})$ and $(\Ka_5)$-$(\Ka_6)$ are fulfilled. Then the set ${\mathcal S}^1_{\!F}(h)$ is an absolute retract for every $h\in C(I,E)$. In other words, the selection set map ${\mathcal S}^1_{\!F}$ is an AR-valued multimap. Moreover, the selection set map ${\mathcal S}^1_{\!F}$ admits a continuous singlevalued selection.
\end{theorem}
\begin{proof}
Observe that the Banach spaces $\left(C(I,E),||\cdot||\right)$ and $\left(L^1(I,E),|||\cdot|||_1\right)$, with $|||\cdot|||_1$ given by \eqref{norm}, are separable. The set-valued map $\left(C(I,E),||\cdot||\right)\times\left(L^1(I,E),|||\cdot|||_1\right)\ni(h,u)\mapsto N_F^1(h+V(u))\subset\left(L^1(I,E),|||\cdot|||_1\right)$ has nonempty closed bounded and decomposable values. In view of Lemma \ref{contraction}. it is continuous and contractive with respect to the second variable. It follows by \cite[Th.1.]{bres} that there exists a continuous function $g\colon C(I,E)\times L^1(I,E)\to L^1(I,E)$ such that $g(h,u)\in\fix\left(N_F^1(h+V(\cdot))\right)$ for all $u\in L^1(I,E)$ and $g(h,u)=u$ for all $u\in\fix\left(N_F^1(h+V(\cdot))\right)$. Therefore, the set ${\mathcal S}^1_{\!F}(h)$ is an AR-space as a retract of the normed space $L^1(I,E)$. With the second argument fixed, the function g becomes a continuous selection of the selection set map ${\mathcal S}^1_{\!F}$.
\end{proof}
\begin{theorem}\label{6}
Let $E$ be a separable Banach space and $p\in(1,\infty)$. Assume that $F$ has convex values and satisfies $({\bf H})$. Let $k$ satisfy conditions $(\Ka_5)$-$(\Ka_6)$. Then the selection set map ${\mathcal S}^p_{\!F}$ is an AR-valued multimap.
\end{theorem}
\begin{proof}
In the context of the assumptions and content of Lemma \ref{contraction}. the set-valued operator $G_p(h,\cdot)\colon\left(L^p(I,E),|||\cdot|||_p\right)\map\left(L^p(I,E),|||\cdot|||_p\right)$, given by $G_p(h,\cdot)(u)=N_F^p(h+V(u))$ is contractive, with convex closed values. Now Theorem 1. in \cite{riccieri} applied to the space $L^p(I,E)$ with the norm $|||\cdot|||_p$ gives the desired result.
\end{proof}
\begin{remark}
There is a strong presumption that the thesis of Theorem \ref{ar}. remains true also for exponents $p>1$. This obviously depends on whether the results of papers \cite{bres, brescol} can be expressed in the setting of $p$-integrable maps.
\end{remark}
\begin{corollary}
Let $E$ be a separable Banach space and $p\in[1,\infty)$. Assume that $F$ satisfies $({\bf H})$ and $F$ has convex values for $p\in(1,\infty)$. Then
\begin{itemize}
\item[(i)]  the solution set $S_{\!F}^p(h)$ is arcwise connected, provided that $k$ satisfies conditions $(\Ka_5)$-$(\Ka_6)$,
\item[(ii)]  the solution set $S_{\!F}^p(h)$ is an absolute retract, provided that $k$ satisfies conditions $(\Ka_1)$-$(\Ka_4)$.
\end{itemize}
\end{corollary}
\begin{proof}
Ad (i): Observe that the solution set $S^p_{\!F}(h)$ and the continuous (affine) image $h+V\left({\mathcal S}^p_{\!F}(h)\right)$, of the arcwise connected set ${\mathcal S}^p_{\!F}(h)$, coincide.\par Ad (ii): Consider $V\colon L^p(I,E)\to\im V\subset C(I,E)$. Conditions $(\Ka_1)$-$(\Ka_4)$ imply the thesis of \cite[Lem.2.]{pietkun}. Let $||\cdot||_{\im}\colon\im V\to\R{}_+$ be such that $||V(w)||_{\im}:=\max\{||V(w)||,||w||_p\}$. Since $V$ is a monomorphism, the mapping $||\cdot||_{\im}$ determines the norm on the subspace $\im V$. Of course, the normed space $(\im V,||\cdot||_{\im})$ is complete. In view of the bounded inverse theorem the inverse map $V^{-1}\colon(\im V,||\cdot||_{\im})\to(L^p(I,E),||\cdot||_p)$ is continuous.\par Consider metric spaces $\left({\mathcal S}^p_{\!F}(h),d\right)$ and $\left(S^p_{\!F}(h),\rho\right)$ with the metrics $d,\rho$ given by the formulae: $d(w,u):=||w-u||_p$, $\rho(x,y):=||x-y||_{\im}$. Then it is obvious that the mapping $V_h\colon\left({\mathcal S}^p_{\!F}(h),d\right)\to\left(S^p_{\!F}(h),\rho\right)$, such that $V_h(w):=h+V(w)$, is continuous. Define the inverse function $V_h^{-1}\colon\left(S^p_{\!F}(h),\rho\right)\to\left({\mathcal S}^p_{\!F}(h),d\right)$ in the following way $V_h^{-1}(x):=V^{-1}(x-h)$. This map is also continuous, due to the continuity of $V^{-1}$. In conclusion, the solution set $S^p_{\!F}(h)$ of the integral inclusion \eqref{inclusion} must be an absolute retract as an homeomorphic image of another absolute retract ${\mathcal S}^p_{\!F}(h)$.
\end{proof}

\end{document}